\documentclass[12pt,a4paper]{amsart}
\usepackage{mathrsfs}

\newtheorem{theo+}              {Theorem}           [section]
\newtheorem{prop+}  [theo+]     {Proposition}
\newtheorem{coro+}  [theo+]     {Corollary}
\newtheorem{lemm+}  [theo+]     {Lemma}
\newtheorem{exam+}  [theo+]     {Example}
\newtheorem{rema+}  [theo+]     {Remark}
\newtheorem{defi+}  [theo+]     {Definition}
\newtheorem{clai+}  [theo+]     {Claim}
\newenvironment{theorem}{\begin{theo+}}{\end{theo+}}
\newenvironment{proposition}{\begin{prop+}}{\end{prop+}}
\newenvironment{corollary}{\begin{coro+}}{\end{coro+}}

\usepackage{amsthm}
\theoremstyle{plain} \theoremstyle{remark}
\newtheorem{remark}{Remark}
\newtheorem{example}{Example}

\def \r{\mbox{${\mathbb R}$}}
\def\E{/\kern-1.0em \equiv }

\author{Ze-Ping Wang$^{*}$, Ye-Lin Ou$^{**}$ and Yong-Gui Luo$^{*}$}

\address{Department of Mathematics,\newline\indent Guizhou
Normal University,\newline\indent Guiyang 550025,\newline\indent
People's Republic of China
\newline\indent E-mail:zpwzpw2012@126.com \;(Wang) \\\newline\indent luoyonggui851010@hotmail.com\;(Luo)\\\newline\indent  \newline\\
\newline\indent Department of
Mathematics,\newline\indent Texas A $\&$ M University-Commerce,
\newline\indent Commerce TX 75429,\newline\indent USA.\newline\indent
E-mail:yelin.ou@tamuc.edu \;(Ou)}
\thanks{*Supported by the Natural Science Foundation of China (No. 11861022). \\
\indent** Supported by a grant from the Simons Foundation ( 427231,
Ye-Lin Ou).}
\date{2/17/2023}
\begin{document}
\title[Harmonic Riemannian submersions]{Harmonic Riemannian submersions from 3-dimensional geometries}

\subjclass{58E20, 53C43} \keywords{Harmonic map, harmonic morphisms, Riemannian
submersions, 3-dimensional BCV spaces, 3-dimensional manifolds.}

\maketitle

\section*{Abstract}
\begin{quote}
{\footnotesize  In this paper, we study harmonic Riemannian submersions from 3-dimensional geometries using the ( generalized) integrability data associated to an orthonormal frame natural to a Riemannian submersion. We give complete classifications of harmonic Riemannian submersions from Thurston's 3-dimensional geometries, 3-dimensional BCV spaces and  Berger sphere  into a surface. We also give some explicit constructions of these harmonic Riemannian submersions.}

\end{quote}

\maketitle

\section{Introduction and Preliminaries}
All manifolds, maps, tensor fields studied in this paper are
assumed to be smooth unless there is an otherwise statement.\\

Recall that a {\em harmonic map} $\varphi:(M, g)\to (N,
h)$ between Riemannian manifolds is a critical point of the energy functional
\begin{equation}\nonumber
E\left(\varphi,\Omega \right)= \frac{1}{2} {\int}_{\Omega}
\left|{\rm d}\varphi \right|^{2}{\rm d}x.
\end{equation}
 and it is therefore the solution of the corresponding Euler-Lagrange equation (see \cite{BW1,EL1}). This
equation is given by the vanishing of the tension field $\tau(\varphi)={\rm
Trace}_{g}\nabla {\rm d} \varphi$, i.e., $\varphi$
is harmonic  if and only if its tension field $\tau(\varphi)={\rm
Trace}_{g}\nabla {\rm d} \varphi$ vanishes identically, i.e.,
\begin{equation}\notag
\tau(\varphi)={\rm
Trace}_{g}\nabla {\rm d} \varphi=0.
\end{equation}

A  Riemannian submersion is called a {\bf harmonic Riemannian submersion} if the Riemannian submersion is a harmonic map.\\

A fundamental problem in the study of harmonic maps is to classify
all harmonic maps between certain model spaces. It is well known (see \cite{BW1}) that there is an interesting class of harmonic maps called harmonic morphisms which are characterized as horizontally weakly conformal harmonic maps. We also have the following inclusion relations\\

$\{Riemannian \;submersions\} \;\subset\;Conformal \;submersions\}$\\

$\hskip1cm \;\subset\; \{horizontally\; weakly\; conformal\;maps\}.$\\

It follows that harmonic Riemannian submersion is a subclass of harmonic morphisms. We refer a reader to the book \cite{BW1} and the vast references therein for basic concepts, interesting applications and important links of harmonic morphisms.

 In this paper, we study  harmonicity of Riemannian submersions  from 3-dimensional Thurston's 3-dimensional geometries,  BCV 3-spaces and  Berger sphere $S_{\varepsilon}^3$ into a surface.
 One of our motivations is that the definition of Riemannian submersions, in a sense, is the dual notion of  isometric immersions (i.e., submanifolds). There are many interesting examples of harmonic isometric immersions of a surface (i.e., minimal surfaces) into 3-manifolds. For example, we state the following: planes or catenoid in $\r^3$ are minimal surfaces; or there are harmonic embedding of $S^2$ into $S^3$, equipped with arbitrary metric \cite{Sm}. On the other hand, there are many examples and classification results for harmonic Riemannian submersions  from 3-dimensional Riemannian manifolds into a surface:
 Hopf fibration $\pi: S^3\to S^2(4)$ and the orthogonal projection $\pi: \r^3\to \r^2$ are harmonic Riemannian submersion (see \cite{BW1,BW2});  there is no harmonic Riemannian submersion
 $\pi:H^3\to(N^2,h)$ no matter what $(N^2,h)$ is (see \cite{WO}); harmonic Riemannian submersions from product spaces$M^2\times\r$ and Sol space have been completely classified (see \cite{WO3,WO4} for details).\\

In this paper, we  study harmonic Riemannian submersions from 3-dimensional geometries using the ( generalized) integrability data associated to an orthononmal frame natural to a Riemannian submersion. We obtain a complete classifications of harmonic Riemannian submersions from Thurston's 3-dimensional geometries, 3-dimensional BCV spaces and  Berger sphere  into a surface. We also give some explicit constructions of these harmonic Riemannian submersions.

\section{Harmonic Riemannian submersions from 3-manifolds}

Let $\pi:(M^3,g) \to (N^2,h)$ be a Riemannian
submersion. Following \cite{WO}, we will call an orthonormal frame  $\{e_1, e_2, e_3\}$ of the total space a {\bf natural orthonormal frame} if $e_1, e_2$ are horizontal and $e_3$ is vertical. If $e_1, e_2$ are basic, then we call  $\{e_1, e_2, e_3\}$ an {\bf adapted frame}  to Riemannian submersion.

With respect to a natural orthonormal frame  $\{e_1, e_2, e_3\}$  we have (see \cite{WO1}) the following  Lie brackets \begin{equation}\label{Le}
[e_1,e_3]=f_{3}e_2+\kappa_1e_3,\;
[e_2,e_3]=-f_{3}e_1+\kappa_2e_3,\;
[e_1,e_2]=f_1 e_1+f_2e_2-2\sigma e_3.
\end{equation}
with  the functions $\{f_1, f_2, f_3, \kappa_1,\kappa_2, \sigma \}$ defined on the total space called  the  generalized integrability data.

Note  (see \cite{WO1})  also that the natural frame $\{e_1,\; e_2, \;e_3\}$ is  adapted to the Riemannian submersion  if and only if $f_{3}=0$.\\

A further computation gives the terms of the curvature of the total space as
\begin{equation}\label{RC}
\begin{cases}
R^{M}(e_1,e_3,e_1,e_2)=-e_1(\sigma)+2\kappa_1\sigma,\\
R^{M}(e_1,e_3,e_1,e_3)=e_1(\kappa_1)+\sigma^2-\kappa_{1}^2+\kappa_2f_1,\;\\
R^{M}(e_1,e_3,e_2,e_3)=e_1(\kappa_2)-e_3(\sigma)-\kappa_{1}f_{1}-\kappa_1\kappa_2,\;\\
R^{M}(e_1,e_2,e_1,e_2)=e_1(f_2)-e_2(f_1)-f_{1}^{2}-f_{2}^{2}+2f_{3}\sigma-3\sigma^2,\\
R^{M}(e_1,e_2,e_2,e_3)=-e_2(\sigma)+2\kappa_2\sigma,\\
R^{M}(e_2,e_3,e_1,e_3)=e_2(\kappa_{1})+e_3(\sigma)+\kappa_2 f_2-\kappa_1 \kappa_2,\\
R^{M}(e_2,e_3,e_2,e_3)=\sigma^{2}+e_2(\kappa_2)-\kappa_1f_2- \kappa_2^2.
\end{cases}
\end{equation}
The  Gauss curvature of the base space is given by

\begin{equation}\label{GCB}
K^N=e_1(f_2)-e_2(f_1)-f_1^2-f_2^2+2f_{3}\sigma.
\end{equation}
When $f_{3}=0$, then Gauss curvature of the base space becomes
\begin{equation}\label{GCB1}
K^N=e_1(f_2)-e_2(f_1)-f_1^2-f_2^2.
\end{equation}

Recall that the tension of the Riemannian  submersion $\pi$ is given by
\begin{equation}\notag
\tau(\pi)=\nabla^{\pi}_{e_i}d\pi(e_i)-d\pi(\nabla^{M}_{e_i}e_i)=-d\pi(\nabla^{M}_{e_3}e_3)
=-d\pi(\kappa_1e_1+\kappa_2e_2).
\end{equation}
A Riemannian  submersion $\pi$ is harmonic if and only if  $\tau(\pi)=0$. Thus, we have
\begin{proposition}\label{PH}
A Riemannian submersion $\pi:( M^3 , g)\to (N^2,h)$ is harmonic if and only if  $\kappa_1=\kappa_2=0$.
\end{proposition}

\begin{example}\label{Ex1}
The Riemannian submersion (projection) from a product  space
\begin{equation}\notag
\begin{array}{lll}
\pi  : M^2\times\r=( \r^3 , g=e^{2p(x,y)}{\rm d}x^{2}+{\rm
d}y^{2}+{\rm d}z^2)\to M^2=(\r^2 ,e^{2p(x,y)}{\rm d}x^{2}+{\rm
d}y^{2}),\\
\pi(x,y,z) =(x,y)
\end{array}
\end{equation}
is  a harmonic map.
\end{example}
Clearly, $\{e_1=e^{-p}\frac{\partial}{\partial x},\;e_2=\frac{\partial}{\partial y},\;e_{3}=\frac{\partial}{\partial z}\}$ form an
orthonormal frame on  $(\r^3,g=e^{2p(x,y)}dx^2+dy^2+dz^2)$ adapted to the Riemannian submersion
with $d\pi(e_3)=0$, and \;$d\pi(e_1)=e^{-p}\frac{\partial}{\partial x}$,\;$d\pi(e_2)=\frac{\partial}{\partial y}$  form an
orthonormal frame on
 $(\r^2,e^{2p(x,y)}dx^2+dy^2)$.  A straightforward computation gives the Lie brackets
\begin{equation}\label{g0}
\begin{array}{lll}
[e_1,e_2]=p_ye_1,\;[e_1,e_3]=0,\;\;[e_2,e_3]=0,
\end{array}
\end{equation}
from which we obtain the integrability data as
$f_1=p_y,\;\sigma=f_2=0,\;\;\kappa_1=\kappa_2=0$.
 We apply Proposition \ref{PH} to conclude that
$\pi$ is harmonic. A harmonic map is always biharmonic.\\

By definition,  harmonic maps  are always biharmonic and biharmonic maps include harmonic maps as special cases.
The following example found in \cite{WO1} shows that the following Riemannian submersion is biharmonic, but not harmonic.
\begin{example}(see also \cite{WO1})\label{Ex3}
The Riemannian submersion
\begin{align}\notag
\pi:H^2\times\r=(\r^{3},e^{2y}dx^2+dy^2+dz^2) &\to
(\r^2 ,dy^2 + dz^2),\; \notag \pi(x,y,z) =(y,z)
\end{align}
is not harmonic. Sot, it is a proper biharmonic map.
\end{example}
The following example shows that the Riemannian submersion from Nil space is neither harmonic nor biharmonic.
\begin{example}\label{Ex3}
The Riemannian submersion from Nil space
\begin{equation}\notag
\pi  : ( \r^3 , g_{Nil}={\rm d}x^{2}+{\rm
d}y^{2}+({\rm d}z-x{\rm d}y)^{2})\to (\r^2 ,dx^2 + (1+x^2)^{-2}dz^2),\;
\pi(x,y,z) =(x,z)
\end{equation}
is neither harmonic nor biharmonic.
\end{example}
We can check that $e_1=\frac{\partial}{\partial
x},\;e_2=-\frac{x}{\sqrt{1+x^2}}\frac{\partial}{\partial
y}-\sqrt{1+x^2}\frac{\partial}{\partial z}, \;
e_3=\frac{1}{\sqrt{1+x^2}}\frac{\partial}{\partial y}$ form an
orthonormal frame on Nil space adapted to the Riemannian submersion
with $d\pi(e_3)=0$, and\;$d\pi(e_1)=\frac{\partial}{\partial
x}$, \;$d\pi(e_2)=-\sqrt{1+x^2}\frac{\partial}{\partial z}$ form an
orthonormal frame on
the base space.  We can compute the Lie brackets as
\begin{align*}
[e_1,e_2]=\frac{x}{1+x^2}e_2-\frac{1-x^2}{1+x^2}e_3,\;
[e_1,e_3]=-\frac{x}{1+x^2}e_3,\; \;[e_2,e_3]=0,
\end{align*}
from which we obtain the integrability data  as
$f_1=0,\;f_2=\frac{x}{1+x^2},\;\;\kappa_1=-\frac{x}{1+x^2},\;\;\sigma=\frac{1-x^2}{2(1+x^2)},\;\;\kappa_2=0$.
Since $\kappa_1=-\frac{x}{1+x^2}\not\equiv0$, we apply Proposition \ref{PH} to conclude that
$\pi$ is not harmonic .\\
From Example 1 in \cite{WO}, the Riemannian submersion
$\pi$ is  not biharmonic, either.\\

We give the following proposition which will be extensively used in the rest of the paper.
\begin{proposition}\label{PH1}
A Riemannian submersion $\pi:(M^3,g) \to (N^2,h)$  with a natural  orthonormal frame  $\{e_1, e_2, e_3\}$ and the  generalized integrability data $\{f_1, f_2, f_3, \kappa_1,\kappa_2, \sigma \}$ is harmonic, then its integrability data  solve the system
\begin{equation}\label{RC0}
\begin{cases}
R^{M}(e_1,e_3,e_1,e_2)=-e_1(\sigma),\\
R^{M}(e_1,e_3,e_1,e_3)=\sigma^2,\;\\
R^{M}(e_1,e_3,e_2,e_3)=-e_3(\sigma)=0,\;\\
R^{M}(e_1,e_2,e_1,e_2)=K^N-3\sigma^2,\\
R^{M}(e_1,e_2,e_2,e_3)=-e_2(\sigma),\\
R^{M}(e_2,e_3,e_1,e_3)=e_3(\sigma)=0,\\
R^{M}(e_2,e_3,e_2,e_3)=\sigma^{2},
\end{cases}
\end{equation}
where $K^N=e_1(f_2)-e_2(f_1)-f_{1}^{2}-f_{2}^{2}+2f_{3}\sigma$ is Gauss curvature of the base space.\\

Conversely, if  the integrability data  associated to a natural frame of a Riemannian submersion satisfy $\sigma\neq 0$ and solve the system (\ref{RC0}), then it is harmonic.
\end{proposition}
\begin{proof}
Let $\pi:(M^3,g) \to (N^2,h)$ be a Riemannian submersion with a natural orthonormal frame  $\{e_1, e_2, e_3\}$ and the  generalized integrability data $\{f_1, f_2, f_3, \kappa_1,\kappa_2, \sigma \}$. It follows from Proposition \ref{PH} that the Riemannian  submersion $\pi$ is harmonic if and only if  $\kappa_1=\kappa_2=0$. Substituting this into (\ref{RC}), it is not difficult to check that the 1st, the 2nd, the 4th, the 5th and the 7th  equation of (\ref{RC0}) hold.  We can also obtain the 3rd  and  the 6th equation  of  (\ref{RC}). Indeed,  substituting $\kappa_1=\kappa_2=0$ into the 3rd  and  the 6th equation  of  (\ref{RC}) separately, we have $R^{M}(e_1,e_3,e_2,e_3)=-e_3(\sigma)$ and $ R^{M}(e_2,e_3,e_1,e_3)=e_3(\sigma)$. However, $R^{M}(e_1,e_3,e_2,e_3)=R^{M}(e_2,e_3,e_1,e_3)={\rm Ric}(e_1,e_2)$, i.e., $-e_3(\sigma)=e_3(\sigma)$, which implies $e_3(\sigma)=0$. Then, the 3rd equation and  the 6th equation  of the system (\ref{RC0}) also hold.\\

Conversely, if $\sigma\neq0$, comparing the 1st equation  of  (\ref{RC0}) with  the 1st equation  of  (\ref{RC}), we obtain $2\kappa_1\sigma=0$ and hence $\kappa_1=0$. Similarly, comparing the 5th equation  of  (\ref{RC0}) with  the 5th equation  of  (\ref{RC}), we have $2\kappa_2\sigma=0$ and hence $\kappa_2=0$. This implies that
 the Riemannian submersion $\pi$ is harmonic by Proposition \ref{PH}.\\
 From which, the proposition follows.
\end{proof}
\begin{remark}\label{r1}
(a) We would like to point out that the condition $\sigma\neq0$  is necessary in the second statement of Proposition \ref{PH1}. For example,  we can check that a Riemannian submersion $\pi:(\r^3, {\rm d}\rho^2+{\rm d}z^2+\rho^2{\rm d}\theta^2)\to(\r^2,{\rm d}\rho^2+{\rm d}z^2)$, $\pi(\rho, z, \theta)=(\rho, z)$  from $\r^3$ into $\r^2$ with  $\sigma=0$  is not harmonic. But, (\ref{RC0}) holds\\

Note that the orthonormal frame  $\{e_1=\frac{\partial}{\partial\rho},\; e_2=\frac{\partial}{\partial z}, \;e_3=\frac{1}{\rho}\frac{\partial}{\partial\theta}\}$ on $(\r^3,{\rm d}\rho^2+{\rm d}z^2+\rho^2{\rm d}\theta^2)$
is adapted to the Riemannian submersion $\pi$ with
${\rm d}\pi(e_i ) = \varepsilon_i , i = 1,2$ and $e_3$ being vertical, where $\{\varepsilon_1=\frac{\partial}{\partial\rho},\; \varepsilon_2=\frac{\partial}{\partial z}\}$ , form an orthonormal frame on the base space $(\r^2,{\rm d}\rho^2+{\rm d}z^2)$. A straightforward computation gives the Lie
brackets
\begin{equation}\notag
[e_1, e_3] = -\frac{1}{\rho} e_3, [e_2, e_3] = 0, [e_1, e_2] = 0.
\end{equation}
The integrability data of the Riemannian submersion $\pi$ are given by
\begin{equation}\label{pr1}
f_1=f_2=\sigma=\kappa_2 = 0,\; \kappa_1=-\frac{1}{\rho}\neq0.
\end{equation}
Since $\kappa_1=-\frac{1}{\rho}\neq0$, then the Riemannian submersion $\pi$ is not harmonic. However,
one can easily compute that $e_1(\kappa_1)=e_1(-\frac{1}{\rho})=\frac{1}{\rho^2}=\kappa_1^2$ and $e_2(\kappa_1)=0$. Substituting this and (\ref{pr1}) into
(\ref{RC}), we find that (\ref{RC0}) holds.\\

(b) It is a interesting fact that the  Riemannian submersion (projection) $\pi:(\r^3, {\rm d}\rho^2+{\rm d}z^2+\rho^2{\rm d}\theta^2)\to(\r^2, {\rm d}\rho^2+\rho^2{\rm d}\theta^2)$, $\pi(\rho, z, \theta)=(\rho, \theta)$  from $\r^3$ into $\r^2$ with  $\sigma=0$  is  harmonic (see \cite{BW1,BW2} for details).
\end{remark}

\section{ Harmonic Riemannian submersions from BCV spaces and Thurston's 3-dimensional geometries}

We first study harmonic Riemannian submersions from the  Bianchi-Cartan-Vranceeanu 3-diemnsional spaces ( 3-dimensional BCV spaces)  which can  be described as (see, e.g., \cite{BDI}):\\
$$M^3_{m,\;l}=(\r^{3},g=\frac{dx^2+dy^2}{[1+m(x^2+y^2)]^2}+[dz+\frac{l}{2}\frac{y
dx-x dy}{1+m(x^2+y^2)}]^2).$$
 We have a globally defined orthonormal frame
\begin{equation}\label{j}
E_{1}=F\frac{\partial}{\partial
x}-\frac{ly}{2}\frac{\partial}{\partial z},\;E_{2}=F
\frac{\partial}{\partial y}+\frac{lx}{2}\frac{\partial}{\partial
z},\;E_{3}=\frac{\partial}{\partial z},
\end{equation}
where $F=1+m(x^2+y^2)$. It is also well known that  BCV
3-spaces  consist of: two of
3-diemnsional space forms $\r^3, S^3$,
the product spaces: $S^2\times\r, H^2\times\r$, and  $\widetilde{SL}(2,\r), Nil, SU(2)$.
A direct computation:
\begin{equation}\label{Lie}
[E_1,E_2]=2mxE_{2}-2myE_{1}+lE_{3},\;{\rm  all\;
other}\;[E_i,E_j]=0,\;i,j=1,2,3.
\end{equation}

Let $\nabla$ denote the Levi-Civita connection of  BCV 3-spaces,
then we have
\begin{equation}\label{BCV1}
\begin{cases}
\nabla_{E_{1}}E_{1}=2myE_{2},\;\;\nabla_{E_{2}}E_{2}=2mxE_{1},\;\;\\
\nabla_{E_{1}}E_{2}=-2myE_{1}+\frac{l}{2}E_{3},\;\;
\nabla_{E_{2}}E_{1}=-2mxE_{2}-\frac{l}{2}E_{3},\;\;\\
\nabla_{E_{3}}E_{1}=\nabla_{E_{1}}E_{3}=-\frac{l}{2}E_{2},\;\;\nabla_{E_{3}}E_{2}=\nabla_{E_{2}}E_{3}=\frac{l}{2}E_{1},\;\;\\
\;\; all \;\; other\;\;
\nabla_{E_i}E_j=0,\;i,j=1,2,3.\;\;\\
\end{cases}
\end{equation}
The possible nonzero components of the
curvatures are given by
\begin{equation}\label{BCV2}
\begin{array}{lll}
 R_{1212}=g(R(E_{1},E_{2})E_{2},E_{1})=4m-\frac{3l^2}{4},\\
R_{1313}=g(R(E_{1},E_{3})E_{3},E_{1})=R_{2323}=g(R(E_{2},E_{3})E_{3},E_{2})=\frac{l^2}{4}.
\end{array}
\end{equation}
and the Ricci curvature:
\begin{equation}\label{BCV3}
\begin{array}{lll}
 {\rm Ric}\, (E_{1},E_{1})={\rm Ric}\, (E_{2},E_{2})=4m-\frac{l^2}{2}\\
  {\rm Ric}\,
(E_{3},E_{3})=\frac{l^2}{2},\;
{\rm all\;other }\; {\rm
Ric}\, (E_i,E_j)=0,\;i\neq j.
\end{array}
\end{equation}

\begin{remark}\label{r2}
By (\ref{j}), (\ref{Lie}) and (\ref{BCV1}),  it is not difficult to find that a map $\pi:M^3_{m,\;l}=(\r^{3},g=\frac{dx^2+dy^2}{[1+m(x^2+y^2)]^2}+[dz+\frac{l}{2}\frac{y
dx-x dy}{1+m(x^2+y^2)}]^2)\to (\r^2, h=\frac{dx^2+dy^2}{[1+m(x^2+y^2)]^2})$ with $\pi(x,y,z)=(x,y)$ is a Riemannian submersion and harmonic.
Note that the orthonomal frame $\{E_{1}=F\frac{\partial}{\partial
x}-\frac{ly}{2}\frac{\partial}{\partial z},\;E_{2}=F
\frac{\partial}{\partial y}+\frac{lx}{2}\frac{\partial}{\partial
z},\;E_{3}=\frac{\partial}{\partial z}\}$, where $F=1+m(x^2+y^2)$, is  adapted to Riemannian submersion $\pi$ and $E_3$ being vertical, and also Gauss curvature of the base space $K^N=4m$.
\end{remark}
 Now we are ready to give the following classification of harmonic Riemannian submersions from BCV 3-spaces.

\begin{theorem}\label{HTh}
A harmonic Riemannian submersion $\pi:M^3_{m,\;l}=(\r^{3},g=\frac{dx^2+dy^2}{[1+m(x^2+y^2)]^2}+[dz+\frac{l}{2}\frac{y
dx-x dy}{1+m(x^2+y^2)}]^2)\to (N^2,h)$  from a BCV space exists only in the case where the target surface has constant Gauss curvature  $K^N=4m$. Furthermore, such harmonic Riemannian submersions
 can be locally described as $\pi:M^3_{m,\;l}=(\r^{3},g=\frac{dx^2+dy^2}{[1+m(x^2+y^2)]^2}+[dz+\frac{l}{2}\frac{y
dx-x dy}{1+m(x^2+y^2)}]^2)\to (\r^2, h=\frac{dx^2+dy^2}{[1+m(x^2+y^2)]^2})$ with $\pi(x,y,z)=(x,y)$\; ( up to equivalence).
\end{theorem}

\begin{proof}
Let $\pi:M^3_{m,\;l}\to (N^2,h)$  from BCV 3-space be a Riemannian submersion with
n natural orthonormal frame $\{e_1,\; e_2,\;e_3\}$ and the (generalized) integrability data $\{ f_1, f_2, f_3, \kappa_1, \kappa_2, \sigma\}$. Denoting by $e_i=\sum\limits_{j=1}^{3}a_i^jE_j, i=1,2,3$, a straightforward computation using (\ref{BCV1}), (\ref{BCV2}), (\ref{BCV3})\; and \;(\ref{RC}), we have the following equalities
\begin{equation}\label{RC2}
\begin{array}{lll}
R^{M}(e_1,e_3,e_1,e_2)=-a_{2}^{3}a_{3}^{3}R,\;
R^{M}(e_1,e_3,e_1,e_3)=(a_{2}^{3})^2R+\frac{l^2}{4},\\
R^{M}(e_1,e_3,e_2,e_3)=-a_{1}^{3}a_{2}^{3}R,\;
R^{M}(e_1,e_2,e_1,e_2)=(a_{3}^{3})^2R+\frac{l^2}{4},\\
R^{M}(e_1,e_2,e_2,e_3)=a_{1}^{3}a_{3}^{3}R,\;
R^{M}(e_2,e_3,e_1,e_3)=-a_{1}^{3}a_{2}^{3}R,\;
R^{M}(e_2,e_3,e_2,e_3)=(a_{1}^{3})^2R+\frac{l^2}{4},
\end{array}
\end{equation}
where $R=4m-l^2$.\\
 Note that using Proposition \ref{PH}, the Riemannian  submersion $\pi$ is harmonic if and only if $\kappa_1=\kappa_2=0$.  By Proposition \ref{PH1}, then (\ref{RC2}) turns into
\begin{equation}\label{RCH}
\begin{cases}
e_1(\sigma)=a_{2}^{3}a_{3}^{3}R,\\
\sigma^2=(a_{2}^{3})^2R+\frac{l^2}{4},\;\\
e_3(\sigma)=a_{1}^{3}a_{2}^{3}R=0,\;\\
K^N=3\sigma^2+(a_{3}^{3})^2R+\frac{l^2}{4},\\
e_2(\sigma)=-a_{1}^{3}a_{3}^{3}R,\\
e_3(\sigma)=-a_{1}^{3}a_{2}^{3}R=0,\;\\
\sigma^{2}=(a_{1}^{3})^2R+\frac{l^2}{4},
\end{cases}
\end{equation}
where $R=4m-l^2$ and $K^N=e_1(f_2)-e_2(f_1)-f_{1}^{2}-f_{2}^{2}+2f_{3}\sigma$.\\
To obtain the theorem,  we consider the following two cases:\\

Case I: $R=4m-l^2=0$. In this case, using the 2nd  and the 4th equation of (\ref{RCH}), we obtain $\sigma^2=\frac{l^2}{4}$ and
$K^N=l^2$. If $l=0$ and hence $m=0$, it follows that the  potential BCV 3-space  and the potential base space $N^2$ have to be $\r^3$ and $\r^2$,
respectively. Naturally, using Remark \ref{r2}, the harmonic Riemannian  submersion is an orthononmal projection $\pi:\r^3\cong\r^2\times\r\to \r^2$ up to equivalence(see also \cite{BW1,BW2}). Moreover,  the Riemannian  submersion  $\pi:\r^3\to \r^2$ can be locally expressed
as $\pi:(\r^3,dx^2+dy^2+dz^2)\to (\r^2,dx^2+dy^2)$ with $\pi(x,y,z)=(x,y)$ with respect to local coordinate. If  $l\neq0$ and hence $m=4l^2>0$, it follows that the potential BCV 3-space and the potential base space $N^2$ have to be a 3-sphere $S^3(m)$ with constant sectional curvature $m$  and  a 2-sphere $S^2(4m)$ with Gauss curvature $4m$, respectively. Actually, by Remark \ref{r2},  the harmonic Riemannian  submersion  $\pi:S^3(m)\to S^2(4m)$  can be locally represented as $\pi:S^3(m)=(\r^{3},\frac{dx^2+dy^2}{[1+m(x^2+y^2)]^2}+[dz+\frac{l}{2}\frac{y
dx-x dy}{1+m(x^2+y^2)}]^2)\to S^2(4m)=(\r^2,\frac{dx^2+dy^2}{[1+m(x^2+y^2)]^2})$ with $\pi(x,y,z)=(x,y)$ with respect to local coordinates, where $4m=l^2>0$.\\

Case II: $R=4m-l^2\neq0$. In this case,  we must have  $a_3^3=\pm1$, i.e., $E_3$ is vertical. In fact, using the 2nd and the 7th equation of (\ref{RCH}) we have $(a_1^3)^2= (a_2^3)^2$.
On the other hand, using the 3rd and the 6th equation of (\ref{RCH}) we have $a_1^3 a_2^3=0$. This leads to $a_1^3= a_2^3=0$.  Hence $a_3^3=\pm1$, i.e., $E_3$ is vertical. Substituting this into the 2nd and the 4th equation of (\ref{RCH}) separately, we have $\sigma^2=\frac{l^2}{4}\;{\rm and}\;K^N=4m$. It is not difficult to find that the  orthonormal frame $\{e_1=E_1,\; e_2=E_2,\;e_3=E_3\}$ is  adapted to  $\pi$ with $e_3= E_3$  being vertical. Clearly, by Remark \ref{r2},  the Riemannian  submersion  $\pi:M_{m,l}^3\to (N^2,h)$ can be locally expressed as $\pi:M_{m,l}^3=(\r^{3},\frac{dx^2+dy^2}{[1+m(x^2+y^2)]^2}+[dz+\frac{l}{2}\frac{y
dx-x dy}{1+m(x^2+y^2)}]^2)\to N^2=(\r^2,\frac{dx^2+dy^2}{[1+m(x^2+y^2)]^2})$ with $\pi(x,y,z)=(x,y)$ with respect to local coordinates.\\
 Combining Case I and Case II,  we obtain the theorem.\\
\end{proof}
Note that the full classification of BCV 3-spaces is as follows(see \cite{BDI})\\
(a) if $m=l=0$, then $M_{m,l}^3\cong\r^3$;\\
(b): if $4m=l^2>0$, then $M_{m,l}^3\cong S^3(m)\backslash\{\infty\}$;\\
(c): if $m>0$ and $l=0$, then $M_{m,l}^3\cong S^2(4m)\backslash\{\infty\}\times\r$;\\
(d): if $m<0$ and $l=0$, then $M_{m,l}^3\cong H^2(4m)\backslash\{\infty\}\times\r$;\\
(e): if $m>0$ and $l\neq0$, then $M_{m,l}^3\cong SU(2)\backslash\{\infty\}$;\\
(f): if $m<0$ and $l\neq0$, then $M_{m,l}^3\cong \widetilde{SL}(2,\r)$;\\
(g): if $m=0$ and $l\neq0$, then $M_{m,l}^3\cong Nil$.\\

Therefore, applying Theorem \ref{HTh}, we have
\begin{corollary}\label{Coro2}
Any harmonic Riemannian  submersion   $\pi:M_{m,l}^3=(\r^{3},\frac{dx^2+dy^2}{[1+m(x^2+y^2)]^2}+[dz+\frac{l}{2}\frac{y
dx-x dy}{1+m(x^2+y^2)}]^2)\to (N^2,h)$ from BCV 3-spaces can be locally expressed as\\
$(i)$: $\pi:(\r^{3},dx^2+dy^2+dz^2)\to (\r^2,dx^2+dy^2)$ with $\pi(x,y,z)=(x,y)$;\;or\\
$(ii)$: $\pi:S^3(m)=(\r^{3},\frac{dx^2+dy^2}{[1+m(x^2+y^2)]^2}+[dz+\frac{l}{2}\frac{y
dx-x dy}{1+m(x^2+y^2)}]^2)\to S^2(4m)=(\r^2,\frac{dx^2+dy^2}{[1+m(x^2+y^2)]^2})$ with $\pi(x,y,z)=(x,y)$ and Gauss curvature of the base space $K^N=4m$, where $4m=l^2>0$;\;or\\
$(iii)$: $\pi:S^2(4m)\times\r=(\r^{3},\frac{dx^2+dy^2}{[1+m(x^2+y^2)]^2}+dz^2)\to S^2(4m)=(\r^2,\frac{dx^2+dy^2}{[1+m(x^2+y^2)]^2})$ with $\pi(x,y,z)=(x,y)$ and Gauss curvature of the base space $K^N=4m$, where $m>0$ and $l=0$;\;or\\
$(iv)$: $\pi:H^2(4m)\times\r=(\r^{3},\frac{dx^2+dy^2}{[1+m(x^2+y^2)]^2}+dz^2)\to H^2(4m)=(\r^2,\frac{dx^2+dy^2}{[1+m(x^2+y^2)]^2})$ with $\pi(x,y,z)=(x,y)$ and Gauss curvature of the base space $K^N=4m$, where $m<0$ and $l=0$;\;or\\
$(v)$: $\pi:SU(2)=(\r^{3},\frac{dx^2+dy^2}{[1+m(x^2+y^2)]^2}+[dz+\frac{l}{2}\frac{y
dx-x dy}{1+m(x^2+y^2)}]^2)\to S^2(4m)=(\r^2,\frac{dx^2+dy^2}{[1+m(x^2+y^2)]^2})$ with $\pi(x,y,z)=(x,y)$ and  Gauss curvature of the base space $K^N=4m$, where $m>0$ and $l\neq0$;\;or\\
$(vi)$: $\pi:\widetilde{SL}(2,\r)=(\r^{3},\frac{dx^2+dy^2}{[1+m(x^2+y^2)]^2}+[dz+\frac{l}{2}\frac{y
dx-x dy}{1+m(x^2+y^2)}]^2)\to H^2(4m)=(\r^2,\frac{dx^2+dy^2}{[1+m(x^2+y^2)]^2})$ with $\pi(x,y,z)=(x,y)$ and  Gauss curvature of the base space $K^N=4m$, where $m<0$ and $l\neq0$;\;or\\
$(vii)$: $\pi:Nil=(\r^{3},dx^2+dy^2+[dz+\frac{l}{2}(y
dx-x dy)]^2)\to (\r^2,dx^2+dy^2)$ with $\pi(x,y,z)=(x,y)$, where $l\neq0$.
\end{corollary}

We give the following classification of harmonic Riemannian  submersion  from  Thurston's 3-dimensional geometries.
\begin{theorem}\label{T}
With respect to local coordinates, any harmonic Riemannian  submersion   $\pi:T^3\to (N^2,h)$ from Thurston's 3-dimensional geometries can be locally expressed as\\
$(i)$: $\pi:(\r^{3},dx^2+dy^2+dz^2)\to (\r^2,dx^2+dy^2)$ with $\pi(x,y,z)=(x,y)$;\;or\\
$(ii)$: $\pi:S^3(m)=(\r^{3},\frac{dx^2+dy^2}{[1+m(x^2+y^2)]^2}+[dz+\frac{l}{2}\frac{y
dx-x dy}{1+m(x^2+y^2)}]^2)\to S^2(4m)=(\r^2,\frac{dx^2+dy^2}{[1+m(x^2+y^2)]^2})$ with $\pi(x,y,z)=(x,y)$, and Gauss curvature of the base space $K^N=4m$, where $4m=l^2>0$;\;or\\
$(iii)$: $\pi:S^2(4m)\times\r=(\r^{3},\frac{dx^2+dy^2}{[1+m(x^2+y^2)]^2}+dz^2)\to S^2(4m)=(\r^2,\frac{dx^2+dy^2}{[1+m(x^2+y^2)]^2})$ with $\pi(x,y,z)=(x,y)$, and Gauss curvature of the base space $K^N=4m$, where $m>0$ and $l=0$;\;or\\
$(iv)$: $\pi:H^2(4m)\times\r=(\r^{3},\frac{dx^2+dy^2}{[1+m(x^2+y^2)]^2}+dz^2)\to H^2(4m)=(\r^2,\frac{dx^2+dy^2}{[1+m(x^2+y^2)]^2})$ with $\pi(x,y,z)=(x,y)$, and Gauss curvature of the base space $K^N=4m$, where $m<0$ and $l=0$;\;or\\
$(v)$: $\pi:\widetilde{SL}(2,\r)=(\r^{3},\frac{dx^2+dy^2}{[1+m(x^2+y^2)]^2}+[dz+\frac{l}{2}\frac{y
dx-x dy}{1+m(x^2+y^2)}]^2)\to H^2(4m)=(\r^2,\frac{dx^2+dy^2}{[1+m(x^2+y^2)]^2})$ with $\pi(x,y,z)=(x,y)$, and Gauss curvature of the base space $K^N=4m$, where $m<0$ and $l\neq0$;\;or\\
$(vi)$: $\pi:Nil=(\r^{3},dx^2+dy^2+[dz+\frac{l}{2}(y
dx-x dy)]^2)\to (\r^2,dx^2+dy^2)$ with $\pi(x,y,z)=(x,y)$, where $l\neq0$.
\end{theorem}
\begin{proof}
It is well known that Thurston's eight models for 3-dimensional geometries consist of: 3-dimensional space forms $S^3$, $\r^3$, $H^3$, the product spaces: $S^2\times\r$, $H^2\times\r$,  and $\widetilde{SL}(2, R)$, $Nil, Sol$. We see that BCV 3-sppaces include six of Thurston's eight 3-dimensional geometries in the family with the exceptions of $H^3$ and Sol space. However, it follows from the results  in \cite{WO} and \cite{WO4} that
there is no a harmonic Riemannian  submersion  from $H^3$, or Sol space to a surface. From these and using Corollary \ref{Coro2}, we get the theorem.
\end{proof}
\begin{remark}\label{r4}
Naturally, a harmonic map is a biharmonic map. We would like to point out that  any  Riemannian  submersion  from $H^3$ or Sol to a surface is neither harmonic nor biharmonc (see \cite{WO},\cite{WO4})
\end{remark}

\section{ Harmonic Riemannian submersion from  Berger 3-sphere $S_{\varepsilon}^3$}
The Hopf map  $ \psi: S^3
(1) \to S^2(4)$ given by
\begin{equation}\label{cmc-1}
\begin{array}{lll}
\psi(x^1,x^2,x^3,x^4)= \frac{1}{2}
(2x^1x^3 + 2x^2x^4, 2x^2x^3- 2x^1x^4,(x^1)^2 + (x^2)^2-(x^3)^2 -(x^4)^2),
\end{array}
\end{equation}
or,
\begin{equation}\label{cmc0}
\begin{array}{lll}
\psi(z, w) = \frac{1}{2}(2zw, |z|^2 - |w|^2) ,
\end{array}
\end{equation}
is a Riemmanian submersion with
totally geodesic $\psi^{-1}(\psi(z, w))$ being the great circle passing through $(z, w)$ and
$(iz, iw)$(see \cite{B}), where $z = x^1 + ix^2$, $w = x^3 + ix^4$ and $S^2(4)$ denotes a 2-sphere with constant Gauss curvature $4$  (i.e., spherical radius $\frac{1}{2}$).
Obviously,  the map $\psi$ is a harmonic Riemmanian submersion.\\

Performing the
following biconformal change of metric to the canonical metric $g$ on $S^3$ with respect to the Hopf map as follows :
\begin{equation}\label{cmc1}
\begin{array}{lll}
g_\varepsilon|_{ T^HS^3\times T^HS^3} =g|_{ T^HS^3\times T^HS^3} , g_\varepsilon|_{ T^VS^3\times T^VS^3}  = \varepsilon^2g, g_\varepsilon|_{ T^HS^3\times T^VS^3} = 0,
\end{array}
\end{equation}
where $T ^V S^3$
and $T^HS^3$ denote  the vertical and the horizontal
spaces determined by $\psi$, respectively. We call a sphere $S^3$ a Berger
3-sphere if the sphere endowed with the metric $g_\varepsilon$  and denote by $S^3_\varepsilon$, that is, $S^3_\varepsilon=(S^3,g_\varepsilon)$, where $\varepsilon\neq0$.\\
For $x \in S^3$, it is not difficult to check that\\ (i): the vector fields
\begin{equation}\label{cmc2}
\begin{array}{lll}
X_1(x) = (-x^2, x^1, -x^4, x^3),\;
 X_2(x) = (-x^4, -x^3, x^2, x^1),\;
X_3(x) = (-x^3, x^4, x^1, -x^2)
\end{array}
\end{equation}
parallelize $S^3$,\\
(ii) $X_1$ is tangent to the fibres of the Hopf map (i.e. $d\psi(X_1) = 0$),\\
and\\
(iii)$ X_2$ and $X_3$ are horizontal, but not bacsic.\\

From (\ref{cmc1}) we immediately deduce that the global frame field
\begin{equation}\label{cmc3}
\begin{array}{lll}
\{E_1 = X_2, E_2 = X_3, E_3 =\varepsilon^{-1} X_1\}
\end{array}
\end{equation}
is an orthonormal frame field on $S^3_\varepsilon$. \\

We adopt the following notation and sign convention for Riemannian
curvature operator:
\begin{equation}
 R(X,Y)Z=\nabla_{X}\nabla_{Y}Z
-\nabla_{Y}\nabla_{X}Z-\nabla_{[X,Y]}Z,\\
\end{equation}
and the Riemannian and the Ricci curvatures:
\begin{equation}
\begin{array}{lll}
&&  R(X,Y,Z,W)=g( R(Z,W)Y,X),\\
&& {\rm Ric}(X,Y)= {\rm Trace}_{g}R=\sum\limits_{i=1}^3 R(Y, e_i, X,
e_i)=\sum\limits_{i=1}^3 \langle R( X,e_i) e_i, Y\rangle.
\end{array}
\end{equation}
In the given frame field, a straightforward computation shows that
\begin{equation}\label{Lieb}
[E_1,E_2]=2 \varepsilon E_{3},\;\; [E_2,E_3]=\frac{2}{\varepsilon}E_1,\; [E_3,E_1]=\frac{2}{\varepsilon}E_2.
\end{equation}
The Levi-Civita connection
of the metric $g_\varepsilon$  has the expression
\begin{equation}\label{g1}
\begin{cases}
\nabla_{E_{1}}E_{1}=0,\;\;\nabla_{E_{1}}E_{2}=\varepsilon E_{3},\;\;\nabla_{E_{1}}E_{3}=-\varepsilon E_{2},\\
\nabla_{E_{2}}E_{1}=-\varepsilon E_{3},\;\;\nabla_{E_{2}}E_{2}=0,\;\;\nabla_{E_{2}}E_{3}=\varepsilon E_{1},\\
\nabla_{E_{3}}E_{1}=\frac{2-\varepsilon^2}{\varepsilon}E_{2},\;\;\nabla_{E_{3}}E_{2}=-\frac{2-\varepsilon^2}{\varepsilon} E_{1},\;\;
\nabla_{E_{3}}E_{3}=0.
\end{cases}
\end{equation}

A further computation (see also \cite{B}) gives the possible
nonzero components of the curvatures:
\begin{equation}\label{g2}
\begin{array}{lll}
 R_{1212}=g(R(E_{1},E_{2})E_{2},E_{1})=4-3\varepsilon^2,\\
R_{1313}=g(R(E_{1},E_{3})E_{3},E_{1})=R_{2323}=g(R(E_{2},E_{3})E_{3},E_{2})=\varepsilon^2,\\
{\rm all\;other }\; R_{ijkl}=g(R(E_{k},E_{l})E_{j},E_{i})=0,\;i,j,k,l=1,2,3.
\end{array}
\end{equation}
 and the
Ricci curvature:
\begin{equation}\label{g3}
\begin{array}{lll}
 {\rm Ric}\, (E_{1},E_{1})={\rm Ric}\,
(E_{2},E_{2})=4-2\varepsilon^2,\;\;\\{\rm Ric}\,
(E_{3},E_{3})=2\varepsilon^2,\; {\rm all\;other }\; {\rm
Ric}\, (E_i,E_j)=0,\;i\neq j.
\end{array}
\end{equation}

\begin{remark}\label{r3}
From (\ref{Le}), (i), (ii), (iii), (\ref{cmc3}), (\ref{Lieb}) and (\ref{g1}),  we can easily check that\\
the  map $ \psi: S^3_{\varepsilon} \to S^2(4)$ given by
$$\pi(x^1,x^2,x^3,x^4)= \frac{1}{2}(2x^1x^3 + 2x^2x^4, 2x^2x^3- 2x^1x^4,(x^1)^2 + (x^2)^2-(x^3)^2 -(x^4)^2),$$\;or,
$$\psi(z, w) = \frac{1}{2}(2zw, |z|^2 - |w|^2)$$
is a Riemmanian submersion with
totally geodesic and $E_3$ being vertical, where $z = x^1 + ix^2$, $w = x^3 + ix^4$.  It is actually a harmonic Riemmanian submersion.
\end{remark}

Now we are ready to give the following classification of harmonic Riemannian submersions from  Berger 3-sphere $S^3_{\varepsilon}$.

\begin{theorem}\label{HTh}
A harmonic Riemannian  submersion $\pi:S_{\varepsilon}^3\to (N^2,h)$ from a Berger 3-sphere to a surface exists only in $S^3_{\varepsilon}\to S^2(4)$ and can be  expressed as $\pi:S^3_{\varepsilon}\to S^2(4)$ with $\pi(x^1,x^2,x^3,x^4)= \frac{1}{2}(2x^1x^3 + 2x^2x^4, 2x^2x^3- 2x^1x^4,(x^1)^2 + (x^2)^2-(x^3)^2 -(x^4)^2),$
or\; $\pi(z, w) = \frac{1}{2}(2zw, |z|^2 - |w|^2)$, where $z = x^1 + ix^2$, $w = x^3 + ix^4$, and  Gauss curvature of the base sphere  $K^N=4$.
\end{theorem}

\begin{proof}
Let $\pi:S^3_{\varepsilon}\to (N^2,h)$  from   Berger 3-sphere  be a Riemannian submersion with
an  orthonormal frame $\{e_1,\; e_2,\;e_3\}$, $e_3$  being vertical, and the generalized integrability data $\{ f_1, f_2, f_3, \kappa_1, \kappa_2, \sigma\}$. We denote by $e_i=\sum\limits_{j=1}^{3}a_i^jE_j, i=1,2,3$. Then, using (\ref{RC}), (\ref{g1}), (\ref{g2}) and (\ref{g3}), we have the following terms of curvature
\begin{equation}\label{RC2b}
\begin{array}{lll}
R^{M}(e_1,e_3,e_1,e_2)=-a_{2}^{3}a_{3}^{3}R,\;
R^{M}(e_1,e_3,e_1,e_3)=(a_{2}^{3})^2R+\varepsilon^2,\\
R^{M}(e_1,e_3,e_2,e_3)=-a_{1}^{3}a_{2}^{3}R,\;
R^{M}(e_1,e_2,e_1,e_2)=(a_{3}^{3})^2R+\varepsilon^2,\\
R^{M}(e_1,e_2,e_2,e_3)=a_{1}^{3}a_{3}^{3}R,\;
R^{M}(e_2,e_3,e_1,e_3)=-a_{1}^{3}a_{2}^{3}R,\;
R^{M}(e_2,e_3,e_2,e_3)=(a_{1}^{3})^2R+\varepsilon^2,
\end{array}
\end{equation}
where $R=4-4\varepsilon^2$.\\
It follows from Proposition \ref{PH} that the Riemannian  submersion $\pi$ is harmonic if and only if $\kappa_1=\kappa_2=0$. By  Proposition \ref{PH1}, then (\ref{RC2b}) becomes

\begin{equation}\label{RCHb}
\begin{cases}
e_1(\sigma)=a_{2}^{3}a_{3}^{3}R,\\
\sigma^2=(a_{2}^{3})^2R+\varepsilon^2,\;\\
e_3(\sigma)=a_{1}^{3}a_{2}^{3}R,\;\\
K^N=3\sigma^2+(a_{3}^{3})^2R+\varepsilon^2,\\
e_2(\sigma)=-a_{1}^{3}a_{3}^{3}R,\\
e_3(\sigma)=-a_{1}^{3}a_{2}^{3}R,\;\\
\sigma^{2}=(a_{1}^{3})^2R+\varepsilon^2,
\end{cases}
\end{equation}
where $R=4-4\varepsilon^2$ and $K^N=e_1(f_2)-e_2(f_1)-f_{1}^{2}-f_{2}^{2}+2f_{3}\sigma$.\\

Noting that if $R=4-4\varepsilon^2=0$,  then the Berger 3-sphere $S^3_{\varepsilon}$ is actually a standard unit sphere $S^3$.  Therefore, it follows
from Theorem 3.1 that  the harmonic Riemannian  submersion $\pi$  from $S^3$  is the  Riemannian  submersion $\pi:S^3\to S^2(4)$ with Gauss curvature of the base sphere $K^N=4$. \\
From now on, we assume that $R=4-4\varepsilon^2\neq0$. In this case,  we must have  $a_3^3=\pm1$,\; i.e., $E_3$ is vertical. In fact,  using the 2nd equation and the 7th equation of (\ref{RCHb}) we have $(a_1^3)^2= (a_2^3)^2$. On the other hand, using the 3rd  and the 6th equation of (\ref{RCHb}) we have $a_1^3 a_2^3=0$. This means $a_1^3= a_2^3=0$. Hence $a_3^3=\pm1$, i.e., $E_3$ is vertical. We substitute this into the 2nd and the 4th equation of (\ref{RCHb}) separately,  to have $\sigma^2=\varepsilon^2\;{\rm and}\;K^N=4$. Using (\ref{g1}), one sees that  $\{e_1=E_1,\; e_2=E_2,\;e_3=E_3\}$ is an  orthonormal frame on $S^3_{\varepsilon}$ with $e_3= E_3$  being vertical and the  generalized integrability data $\{ f_1=f_2= \kappa_1=\kappa_2=0,\;\sigma=\varepsilon^2,\;f_3=\frac{2}{\varepsilon^2}\}$. Obviously, by Remark \ref{r3},  a harmonic Riemannian  submersion  $\pi:S_{\varepsilon}^3\to (N^2,h)$ can be expressed as $\pi:S^3_{\varepsilon}\to S^2(4)$ with
$\pi(x^1,x^2,x^3,x^4)= \frac{1}{2}(2x^1x^3 + 2x^2x^4, 2x^2x^3- 2x^1x^4,(x^1)^2 + (x^2)^2-(x^3)^2 -(x^4)^2),$
or $\pi(z, w) = \frac{1}{2}(2zw, |z|^2 - |w|^2)$, where $z = x^1 + ix^2$, $w = x^3 + ix^4$.\\
 From which the theorem follows.

\end{proof}

\end{document}